\documentclass[11pt]{amsart}
 \usepackage{amsmath,amssymb,enumerate}
 \usepackage[all]{xy}
 \usepackage{xcolor}

 \newtheorem{pro}{Proposition}[section]
 
 \newtheorem{lem}[pro]{Lemma}
 \newtheorem{rem}[pro]{Remark}
 \newtheorem{exa}[pro]{Example}
 
  \newtheorem{defi}[pro]{Definition}
 \newtheorem{theo}[pro]{Theorem}
 \newtheorem{thm}[pro]{Theorem}

 \newcommand{\B}{\mathbb B}
 
 \newcommand{\R}{\mathbb R}
 
 \newcommand{\C}{\mathbb C}

  \newcommand{\Pn}{\mathbb P^{n}}
  
   \newcommand{\G}{\mathbb G}

 \newcommand{\e}{\varepsilon}

 \newcommand{\f}{\varphi}

 \newcommand \mrm{\mathrm}

 \numberwithin{equation}{section}
\usepackage{hyperref}
\begin{document}
\title[Logarithmic Potentials]{ Projective Logarithmic Potentials}
\setcounter{tocdepth}{1}
\author{ S. Asserda, Fatima Z. Assila and  A. Zeriahi}

\address{Universit\'e Ibn Tofail,
K\'enitra Maroc \\} 
\email{$fati_zahra_91$@yahoo.fr} 

\address{Universit\'e Ibn Tofail,
K\'enitra Maroc \\} 
\email{asserda-said@univ-ibntofail.ac.ma}

\address{ IMT; UMR 5219 \\ \newline
Universit\'e de Toulouse; CNRS  \\  \newline
  Paul Sabatier,  F-31400 Toulouse, France} 
 \email{zeriahi@math.univ-toulouse.fr}

\maketitle

 \begin{center}
{\it A tribute to Professor Ahmed Intissar}
\end{center}

\date{\today \\
 The third author was partially supported by the ANR project GRACK}
 \begin{abstract} We study the projective logarithmic potential  $\G_\mu$ of a Probability measure $\mu$ on the complex projective space $\mathbb{P}^{n}$ equiped with the Fubini-Study metric $\Omega$. We prove that the Green operator $\G : \mu \longmapsto \G_\mu$ has strong regularizing  properties.
 
  It was shown by the second author in \cite{As17} that the range of the operator $\G$ is contained in  the (local) domain of definition of the complex Monge-Amp\`ere operator on $\mathbb P^n$. This result extend earlier results by Carlehed \cite{Carlehed99}.
  
   Here we will show that the complex  Monge-Amp\`ere measure 
 $(\omega + dd^c \G_\mu)^n$ of the logarithmic potential of $\mu$  is absolutely continuous with respect to the Lebesgue measure  on $\mathbb P^n$ {\it if and only if} the measure $\mu$  has no atoms. 
 Moreover when the measure $\mu$ has a "positive dimension", we give more precise results on regularity properties of the potential $\G_\mu$ in terms of the  dimension of $\mu$. 
 \end{abstract}
 
\tableofcontents
\section{Introduction}
 In Classical Potential Theory (CPT) various potentials associated to Borel measures in the euclidean space $\R^N$ were introduced. They play a fundamental role in many problems (see  \cite{Carleson65,La72}).  This is due to the fact that CPT is naturally associated to the Laplace operator which is a linear elliptic partial differential operator of second order with constant coefficients. Indeed  any subharmonic function is locally equal (up to a harmonic function) to the Newton  potential of its Riesz measure when $N \geq 3$, while in the case of the complex plane $\C \simeq \R^2$, we need to consider the logarithmic potential.
 
On the other hand, it is well known that in higher complex dimension $n \geq 2$, plurisubharmonic functions are rather connected to the complex Monge-Amp\`ere operator which is a fully non linear second order differential operator. Therefore Pluripotential theory cannot be completely described by logarithmic potentials as shown by Magnus Carlehed in \cite{Carlehed}. However the class of logarithmic potentials provides a natural class of plurisubharmonic functions which turns out to be included in the local domain of definition of the complex Monge-Amp\`ere operator.
This study was started by Carlehed (\cite{Carlehed99, Carlehed} in the case when the measue is compactly supported on $\C^n$ with a locally bounded potential.

Our main goal is to extend this study to the case of arbitrary probability measures on the complex projective space $\mathbb P^n$. The motivation for this study comes from the fact that the complex Monge-Amp\`ere operator plays an important role in K\"ahler geometry when dealing with the Calabi conjecture and the problem of the existence of a K\"ahler-Einstein metric (see \cite{GZ17}).
 Indeed in geometric applications one needs to consider degenenerate complex Monge-Amp\`ere equations. To this end a large class of singular potentials on which the complex Monge-Amp\`ere operator is well defined was introduced  (see \cite{GZ07}, \cite{BEGZ10}). This leads naturally to a  general definition of the complex Monge-Amp\`ere operator (see \cite{CGZ08}). 
  However the global domain of definition of the complex Monge-Amp\`ere operator on compact K\"ahler manifolds is not yet well understood. 
  
  Besides this, thanks to the works of Cegrell  and Blocki (see \cite{Ce04}, \cite{Bl04}, \cite{Bl06}), the local domain of definition  is characterized in terms of some local integrability conditions on approximating sequences. 
  
Our goal here is then to study a natural class of projective logarithmic potentials and show that it is contained in the (local) domain of definition of the complex Monge-Amp\`ere operator on the complex projective space $\mathbb P^n$,  giving an  interesting subclass of the domain of definition of the complex Monge-Amp\`ere operator on the complex projective space and exhibiting many different interesting behaviours.

More precisely, let $\mu$ be  a  probability  measure on the complex projective space $\mathbb{P}^{n}$. Then its projective logarithmic potential is defined on $\mathbb{P}^{n}$ as follows:
$$ 
\G_\mu (\zeta) := \int_{\mathbb P^n }\log \frac{\vert \zeta \wedge \eta \vert}{\vert \zeta \vert \vert \eta \vert}\ d \mu (\eta), \, \, \zeta \in 
\mathbb P^n.
$$

 \smallskip
 \smallskip

 Our first result is a regularizing property of the operator $\G$ acting on the convex compact set of probability measures on $\mathbb P^n$.

\smallskip
 \smallskip
  
{\bf Theorem A}. {\it Let $\mu$ be a probability measure on $\mathbb P^n$ ($n \geq 2$). Then

\smallskip

 1)  $\G_\mu  \in DMA_{loc} (\mathbb P^n,\omega)$ and  for any $0 < p < n$,  $G_\mu \in W^{2,p} (\mathbb P^n)$ ;  

\smallskip

 2)  for any $1 \leq k \leq n - 1$, the $k$-Hessian measure $(\omega +   ddĉ G_\mu)^k \wedge \omega^{n - k}$ is absolutely continous with respect to the Lebesgue measure on $(\mathbb P^n,\omega)$;

\smallskip

 3)  the Monge-Amp\`ere measure   $(\omega + dd^c \G_\mu)^n$ is absolutely continuous with respect to the Lebesgue measure on $(\mathbb P^n,\omega)$ if  the measure $\mu$ has no atom in $\mathbb P^n$; conversely  if the measure $\mu$ has an atom at some  point $a\in \mathbb P^n$, the complex Monge-Amp\`ere measure  $(\omega + dd^c \G_\mu)^n$  also has an atom at the same point $a$.}

\smallskip
\smallskip
 
Here $\omega := \omega_{FS}$ is the Fubini-Study metric on the projective space $\mathbb P^n$ and $ DMA_{loc} (\mathbb P^n, \omega)$ denotes the local domain of definition of the complex Monge-Amp\`ere operator on the K\"ahler manifold $(\mathbb P^n,\omega)$ (see Definition \ref{def:DMA_loc} below).

 \smallskip
 The fact that the projective logarithmic potential  of any measure belongs to the domain of definition of the complex Monge-Amp\`ere operator as well as the third property was proved earlier by the second named author (\cite{As17}). 

 %Moreover in dimension $2$ it is possible to give an explicit formula for the density $g_\mu$.
 These results  generalize and improve previous results by M. Carlehed who considered the local setting (see \cite{Carlehed, Carlehed99}).
 
 \smallskip
 \smallskip
 
 When the measure has a "positive dimension",  we obtain a strong regularity result in terms of this dimension as defined by the formula (\ref{eq:dim}) in section 3.2.

\smallskip
 \smallskip
 
{\bf Theorem B :} {\it Let $\mu$ be a probability measure on $\mathbb P^n$ ($ n \geq 2$) with positive dimension $ \gamma(\mu) > 0 $.
Then  the following holds :

1) $\G_\mu \in W^{1,p} (\mathbb P^n)$ for any $p$ such that $1 \leq p < (2 n - \gamma (\mu)) \slash (1 - \gamma(\mu))_+$; in particular $\G_\mu$ is H\"older continuous of any exponent $\alpha$ such that
$$
0 < \alpha < 1 -   2 n \frac{(1  - \gamma(\mu))_+}{ 2 n - \gamma(\mu)};
 $$

2)  $\G_\mu \in W^{2,p} (\mathbb P^n)$ for any 
$$
 1 < p < \frac{2 n -\gamma(\mu)}{(2 - \gamma(\mu))_+},
$$ 

3)   the  density of the complex Monge-Amp\`ere measure $(\omega + dd^c \G_\mu)^n$ with respect to the Lebesgue measure on $\mathbb P^n$ satisfies $g_\mu \in L^{q} (\mathbb {P}^n)$ for any 
$$
1 < q <    \frac{ 2 n - \gamma(\mu)}{ n (2 - \gamma(\mu))_+} \cdot
$$}

Recall that for a real number $s \in \R$, we set $s_+ := \max\{s,0\}$. Observe that in the last two statements, the cirtical exponent is $+ \infty$ when $ 2  \leq \gamma(\mu) \leq 2n$.

\section{Preliminaries}

\subsection{Lagrange identities}

 Let us fix some notations. For  $a = (a_0, \cdots, a_n) \in \C^{n + 1}$ and $ b =  (b_0, \cdots, b_n) \in \C^{n + 1}$ we set
 $$
 a \cdot \bar b := \sum_{j = 0}^n a_j \bar {b}_j,  \, \, \,  \, \, \vert a\vert^2 :=  \sum_{j = 0}^n \vert a_j \vert^2.
 $$
 We can define the wedge product $a \wedge b$ in the vector space  $\bigwedge^2 \C^{n + 1}$.
 The hermitian scalar product on $\C^{n +1}$ induces a natural hermitian scalar product on $\bigwedge^2 \C^{n + 1}$. If $(e_j)_{0 \leq j \leq n}$ is the associated canonical orthonormal basis of $\C^{n+1}$, then the  sequence $( e_i \wedge e_j)_{0 \leq i < j \leq n}$ is an orthonormal basis of the vector space $\bigwedge^2 \C^{n + 1}$. Therefore we have  
 $$
  a \wedge b = \sum_{0 \leq i < j \leq n}  (a_i b_j - a_j b_i)  \, e_i \wedge e_j,
 $$
and
 $$
 \vert  a \wedge b  \vert^2 = \sum_{0 \leq i <  j \leq n}  \vert a_i b_j - a_j b_i \vert^2.
 $$
 \begin{lem} (Lagrange Identities) 
 
  1. For any $a, b \in \C^{n + 1} \setminus \{0\}$,
 \begin{equation} \label{eq:Lagrange}
   \vert  a \wedge b  \vert^2 = \vert a \vert^2 \vert b\vert^2 - \vert a \cdot \bar b\vert^2.
 \end{equation}
 
 2. For any $a, b \in \C^{n + 1} \setminus \{0\}$,
 \begin{equation} \label{eq:Lagrange1}
   \frac{\vert  a \wedge b  \vert^2}{\vert a \vert^2 \vert b\vert^2 } = 1 - \frac{\vert a \cdot \bar b\vert^2}{\vert a \vert^2 \vert b\vert^2 }.
 \end{equation}
 
 3. For any $z, w \in \C^n$,
    \begin{equation} \label{eq:Lagrange2}
   \frac{\vert z - w\vert^2 + \vert  z \wedge w  \vert^2}{(1 + \vert z \vert^2)(1 +  \vert w\vert^2) } = 1 - \frac{ \vert 1 +  z \cdot \bar w\vert^2}{(1 + \vert z \vert^2)(1 +  \vert w\vert^2)}.
 \end{equation}
 \end{lem}
 \begin{proof} Here  $a \cdot \bar b := \sum_{j = 0 }^n a_j \bar b_j$ is the euclidean hermitian product on $\C^{n + 1}$.
 The first identity is the so called Lagrange identity, while the second follows immediately from the first.

 The third identity follows from the second one applied to the vectors $a = (1,z)$ and $b = (1,w)$.
 \end{proof}

\subsection{The complex projective space}
 Let $\mathbb P^n  = \C^{n + 1} \setminus \{0\} \slash  \C_*$ be the complex projective space of dimension $n \geq 1$ and
 $$
 \pi : \C^{n + 1} \setminus \{0\} \longrightarrow \mathbb P^n,
 $$
 the canonical projection. which sends a point $\zeta = (\zeta_0, \cdots, \zeta_n) \in  \C^{n + 1} \setminus \{0\}$ to the complex line $\C_* \cdot \zeta$ which we denote by $[\zeta] = [\zeta_0, \cdots, \zeta_n]$.

 By abuse of notation we will denote by $\zeta= [\zeta_0, \cdots, \zeta_n]$ and call the $\zeta_j$'s the homogenuous coordinates of the (complex line) $\zeta$.

 As a complex manifold $\mathbb{P}^n$ can be coved by a finite number of charts given by
 $$
 \mathcal{U}_k := \{\zeta \in \mathbb P^n ; \zeta_k \neq 0 \}, \, \, \, \, \,  0 \leq k \leq n.
 $$
 For a fixed $k = 0, \cdots, n$, the corresponding coordinate chart is defined on $\mathcal{U}_k$ by the formula
 $$
 {\bf z}^k (\zeta)  = z^k    := (z^k_j)_{0 \leq j \leq n, j \neq k} , \, \, \mrm{where} \, \, z^k_j := \zeta_j \slash \zeta_k \, \, \mrm{for} \, \,  j \neq k.
 $$
 The map $\gamma_k : \mathcal{U}_k \simeq \C^n$  is an homomeorphism and for $k \neq \ell$ the transition functions (change of coordinates)  
 $$
 {\bf z}^k \circ {( {\bf z}^\ell)}^{-1}  :   {\bf z}^\ell ( \mathcal {U}_\ell \cap  \mathcal {U}_k) \longrightarrow    {\bf z}^k ( \mathcal {U}_\ell \cap  \mathcal {U}_k)
$$
 is given by

 $$
  w =   {\bf z}^\ell \circ {( {\bf z}^k)}^{-1} (z_1, \cdots, z_n), \, \, \, \, \,   \, \,  \mrm{for} \, \, (z_1, \cdots, z_n) \in \C^n, z_\ell \neq 0, 
 $$
 where $w_i = z_i \slash z_\ell$ for $i \notin \{k,\ell\}$, $w_k = 1 \slash z_\ell$ and $w_\ell = z_k \slash z_\ell$.

 The  $(1,1)$-form $dd^c \log \vert \zeta \vert$ is smooth, $d$-closed  on $\C^{n + 1} \setminus \{0\}$ and invariant under the $\C_*$-action. Therefore it descends to $\mathbb P^n$ as a smooth closed $(1,1)$-form $\omega_{FS}$ on $\mathbb P^n$ so that
 $$
  dd^c \log \vert \zeta \vert = \pi^* (\omega_{FS}), \, \, \, \mrm{in} \, \, \C^{n + 1} \setminus \{0\} .
 $$
 In the local chart $(\mathcal U_k,{\bf z}^k)$ we have
 $$
 \omega\mid{\mathcal U_k} = \frac{1}{2} dd^c \log (1 + \vert {z}^k\vert^2)
 $$

 Observe that the Fubini-Study form $\omega$ is a K\"ahler form on $\mathbb P^n$ and the correponding Fubini-Study volume form $d V_{FS} = \omega^n \slash n !$ is given in the chart $(\mathcal U_k,{\bf z}^k) \simeq (\C^n ,z)$ by the formula
 $$
d V_{FS} \mid \C^n  = c_n \frac{ d V_{2 n} (z)}{(1 + \vert z\vert^2)^{n + 1}},
 $$
 where $d V_{2 n} (z) = \beta_n = \beta^n \slash n !$ is the euclidean volume form on $\C^n$,   $\beta := dd^c \vert z\vert^2$ being the standard K\"ahler  metric on $\C^n$.
 
 \subsection{The complex  Monge-Amp\`ere operator} Here we recall some definitions and give a  useful characterization of the local domain of definition of the complex  Monge-Amp\`ere operator given by  Z. B\l ocki (see \cite{Bl04}, \cite{Bl06}).

 \begin{defi} \label{def:DMA_loc}
 Let  $X$ be a complex maniflod of dimension $n$ an $\eta$ a smooth closed (semi)-positive $(1,1)$-form on $X$. 
 
 1) We say that a function $\f : X \longrightarrow \R \cup \{- \infty\}$ is  $\eta$-plurisubharmonic in $X$ if it is locally the sum of a plurisubharmonic function and a smooth function and $\eta + dd^c \f$ is a positive current on $X$. We  denote by $PSH (X,\eta)$ the convex set of all of $\eta$-plurisubharmonic  functions  in $X$. 
 
 2) By definition, the set  $DMA_{Loc}(X,\eta)$ is the set of  functions $\f \in PSH (X,\eta)$   for which
 there exists a positive Borel measure  $\sigma = \sigma_\f$ on $X$ such that  for all  open  $ U\subset\subset \Omega $  and  $\forall (\f_{j})\in PSH(U,\eta)\cap C^{\infty}(U)$ $\searrow$ $\f$ in $U$,  the sequence of local Monge-Amp\`ere measures  $(\eta + dd^{c}\f_{j})^{n}$ converges weakly to $\sigma$ on $U$. In this case, we set $(\eta + dd^{c}\f)^{n}=\sigma_\f$ and call it the complex Monge-Amp\`ere measure of the function $\f$ in the manifold $(X,\eta)$.
 
 When $\eta = {\bf 0}$, we write $DMA_{loc} (X) = DMA_{loc} (X,{\bf 0})$.
 
 \end{defi}
 \noindent In the case when $X \subset \C^n$ is an open subset and $\eta = {\bf 0}$,  Bedford and Taylor \cite{BT76, BT82} extended the complex Monge-Amp\`ere operator $(dd^{c}u)^{n}$ to  plurisubharmonic functions which are locally bounded in $X$. Moreover, they showed that this operator is continuous under decreasing sequences in $PSH (X) \cap L^{\infty}_{Loc} (X)$.

In \cite{De87}, Demailly extended  the complex Monge-Amp\`ere operator $(dd^{c})^{n}$ to plurisubharmonic functions  locally bounded near the boundary $\partial X$ i.e. in the complement of a compact subset of $X$.

When $X = \mathbb P^n$ and $\eta = \omega_{FS}$, it was proved in \cite{CGZ08} that if $\f \in PSH (\mathbb P^n,\omega_{FS})$ is bounded in a neighbourhood of a divisor in $\mathbb P^n$, then $\f \in DMA_{loc} (\mathbb P^n,\omega_{FS})$.

When $\Omega \Subset \C^n$ is a bounded hyperconvex domain, the set $DMA_{loc} (\Omega)$ coincides with the Cegrell class $\mathcal E (\Omega)$ (see \cite{Ce04}, \cite{Bl06}).
Moreover for any $u_1, \cdots, u_n$ $DMA_{loc} (\Omega)$, it is possible to define the intersection current $dd^c u_1 \wedge \cdots \wedge dd^c u_n$ as a positive Borel measure on $\Omega$ (\cite{Ce04}).

Therefore for any $\f_1, \cdots, \f_n \in DMA_{loc} (X,\eta)$, it is possible to define the   intersection current 
$$ \mathcal M (\f_1,\cdots,\f_n) := (\eta + dd^c \f_1) \wedge \cdots \wedge (\eta + dd^c \f_n)
$$ 
as a positive Borel measure on $X$.
In particular if $X$ is compact and $\eta$ is a smooth $(1,1)-$form on $X$ which is closed semi-positive and big i.e. $\int_X \eta^n > 0$, then 
$$
\int_X (\eta + dd^c \f_1) \wedge \cdots \wedge (\eta + dd^c \f_n) = \int_X \eta^n.
$$

Moreover the operator $\mathcal M$ is continuous for  monotone convergence of sequences in $DMA_{loc} (X,\eta)$.

In dimension two, a simple characterization of the local domain of definition of $(dd^{c})^{2}$ was  given by \cite{Bl04}.
\begin{thm}\cite{Bl04}
 Suppose $\Omega$ is an open subset of   $\mathbb{C}^{2}$, then $$DMA_{Loc}(\Omega)=PSH(\Omega)\cap {W}^{1,2}_{loc}(\Omega),$$
 where ${W}^{1,2}_{loc}(\Omega)$ is the usual Sobolev space.
 \end{thm}
 
 In higher dimension, the characterization of the set $DMA_{Loc}(\Omega)$ is more complicated (see \cite{Bl06}).
\begin{thm} \label{thm:Bl}
Let  $u$ be a negative plurisubharmonic function on $\Omega \subset \mathbb{C}^{n}$, $n \geq 3$. The following are equivalent:

1- $u\in DMA_{Loc}(\Omega)$,
   
2- $\forall z\in \Omega $,   $\exists U_{z}\subset \Omega$ an open neighborhood of $z$ such that for  any sequence $ u_{j}\in PSH\cap C^{\infty}(U_{z})\searrow u $ in  $U_{z},$ the sequences
$$
|u_{j}|^{n-p-2}du_{j}\wedge d^{c}u_{j}\wedge (dd^{c}u_{j})^{p}\wedge\omega^{n-p-1},\;\; p=0,1,...,n-2
$$
are locally weakly bounded in  $U_{z}$.
\end{thm}

It is clear that these results extend to the case of $DMA_{loc} (X,\eta)$. 
 
 \begin{rem}
 We can define the global domain of defintion $DMA (X,\eta)$ by taking only bounded global approximants. Therefore $DMA_{loc} (X,\eta) \subset DMA (X,\eta)$ but the inclusion is strict (see \cite{GZ07}) at least in the case when  $(X,\eta)$ is a compact K\"ahler manifold. Very little is known about the space $DMA (X,\eta)$ (see \cite{CGZ08}).
 \end{rem}
 
 \subsection{The Lelong Class}
 Recall that the Lelong class  $\mathcal L (\C^n)$  is defined as the convex set of plurisubharmonic functions in $\C^n$ satisfying the following growth condition at infinity :
 \begin{equation}
  u (z ) \leq C_u + \frac{1}{2} \log (1 + \vert z\vert^2), \, \, \, \forall z \in \C^n,
 \end{equation}
 where $C_u$ is a constant depending on $u$ (see \cite{Le68}).

 There are two important subclasses of $\mathcal L (\C^n)$. Set
 $$
 \mathcal L_+ (\C^n) := \{u \in \mathcal L (\C^n) ; u (z) = \frac{1}{2} \log (1 + \vert z\vert^2) + O (1), \, \text{as} \, \, \vert z\vert  \to + \infty\}.
 $$

 Associated to a function $u \in \mathcal L (\C^n)$ we define its Robin function of $u$ (\cite{BT88})
 $$
 \rho_u (\xi) := \limsup_{ \C \ni \lambda  \to  \infty} (u (\lambda \xi) - \log \vert \lambda \xi\vert) = \limsup_{ \C \ni \lambda  \to  \infty} (u (\lambda \xi) - \frac{1}{2}\log (1 + \vert \lambda \xi\vert^2).
 $$
 When $\rho_u^* \not \equiv  - \infty$, then  it is a homogenous function of order $0$ on $\C^n$ such that $\rho_u^* (z) + \log \vert z\vert$ is plurisubharmonic in $\C^n$. This implies that $\rho_u^*$  is a well defined function on the projective space $\mathbb P^{n - 1}$ which is actually  an $\omega_{FS}$-psh on $\mathbb  P^{n - 1}$, where $\omega_{FS}$ is the Fubini-Study metric on $\mathbb  P^{n - 1}$.

Following (\cite{BT88}),  we can define the subclass of logarithmic  Lelong potentials as follows :
 $$
  \mathcal L_\star (\C^n) := \{ u \in  \mathcal L (\C^n) ; \rho_u \not \equiv - \infty \}\cdot
 $$
 Observe that $ \mathcal L_+ (\C^n) \subset \mathcal L_\star (\C^n) $ and $ u\in \mathcal L_\star (\C^n) $ iff $\rho_u \in L^1 (\mathbb P^{n - 1})$.
 We refer to \cite{Ze07} for more properties of this class.

Then there is a $1$-$1$ correspondence between the Lelong class $\mathcal L (\C^n)$ and the set $PSH (\Pn,\omega)$ of $\omega$-plurisubharmonic functions on  $\mathbb P^n$. Indeed we will write  $\mathbb P^n = \mathcal {U}_0 \,   \dot {\cup} \,  H_{\infty}$, where
$$
 H_{\infty} := \{\zeta \in \mathbb P^n ; \zeta_0 = 0\} \simeq \mathbb P^{n-1},
 $$
 is the hyperplane at infinity and  observe that $\mathcal {U}_0 = \mathbb{P}^n \setminus H_{\infty}$.

 Given $u \in  \mathcal L (\C^n)$, we associate the function $\f$ defined on $\mathcal {U}_0$  by
 $$
 \phi_u  (\zeta) := u (z_1, \cdots , z_n) - \frac{1}{2} \log (1 + \vert z\vert^2), \, \, \, \text{with} \, \,  z : = z^0  = (\zeta_1 \slash \zeta_0, \cdots , \zeta_n \slash \zeta_0).
 $$
By definition of the class $\mathcal L (\C^n)$, the function $\phi_u$ is locally upper bounded in $\mathcal U_0$ near $H_{\infty}$. Since $H_{\infty} = \{0\} \times \mathbb P^{n-1}$ is a proper analytic subset (hence a pluripolar subset) of  $\mathbb P^n$, the function $\phi_u$ can be extended into an $\omega$-plurisubharmonic function on $\mathbb P^n$ by setting
 $$
  \phi_u  (0,\zeta') = \limsup_{ \mathcal U_0 \ni \zeta \to (0,\zeta')} \f (\zeta),  \, \, \, \zeta' \in \mathbb P^{n - 1}.
 $$
  Then we have a well defined "homogenization" map
 $$
   \mathcal{L} (\C^n) \ni \longmapsto \phi_u \in PSH (\mathbb P^n, \omega_{FS}).
 $$
 This is a bijective map and its  inverse map is defined as follows :  given $\phi \in PSH (\mathbb P^n, \omega_{FS}),$ we can define
 $$
 u_\phi (z) := \phi ([1,z]) + (1 \slash 2) \log (1 + \vert z\vert^2).
 $$
 It is clear that $ u_\phi  \in \mathcal{L} (\C^n)$  and $H (u_\phi) = \phi$. Observe that for $\zeta' \in \mathbb{P}^{n -1}$, we have
 $$
  \rho_u  (\zeta') = \phi_u (0,\zeta').
  $$
\begin{exa}
  Let $P \in \C_d [z_1, \cdots,z_n]$ be a polynomial of degree $d \geq 1$. Then $u := (1 \slash d)  \log \vert P\vert  \in
   \mathcal{L} (\C^n)$. It is easy to see  that its homogeneization $\phi = \phi_u$ is given by
   $$
   \phi (\zeta) :=  (1 \slash d)  \log \vert Q (\zeta) \vert - \log \vert \zeta \vert, \, \, \zeta \in \mathbb{P}^n,
$$
where $Q$ is the unique  homogenuous polynomial on $\C^{n + 1 }$ of degree $d$   such that by $Q (1,z) = P (z)$ for $z \in \C^n$.
\end{exa}

\begin{pro} Let $u_1, \cdots, u_n \in \mathcal L_+ (\C^n)$ and for each $i = 1, \cdots, n$ set $\phi_i := \phi_{u_i}$ the homogeneization of $u_i$ on $\mathbb  P^n.$ Then
\begin{equation} \label{eq:totalmass}
\int_{\C^n} dd^c u_1 \wedge \cdots \wedge dd^c u_n  = \int_{\mathbb{P}^n} (\omega + dd^c \phi_1) \wedge \cdots \wedge (\omega + dd^c \phi_n) = 1.
\end{equation}
\end{pro}
\begin{proof} Observe that if $u \in \mathcal L_+ (\C^n)$, its homogeneization  $\phi = \phi_u \in PSH (\mathbb{P}^n,\omega)$  is a bounded $\omega$-psh function in a neighbourhood of the hyperplane at infinity $H_{\infty}$.
Hence by \cite{CGZ08}, $ \mathcal L_+ (\C^n) \subset DMA_{loc} (\mathbb P^n,\omega)$ and its complex Monge-Amp\`ere measure is well defined and puts no mass on $H_{\infty}$. Hence
$$
1 = \int_{\mathbb{P}^n} (\omega + dd^c \phi)^n = \int_{\C^n} (\omega + dd^c \phi)^n.
$$
Since $\omega + dd^c \phi = dd^c u$ in the weak sense on $\C^n$, the formula (\ref{eq:totalmass}) follows.

\end{proof}

\section{Projective logarithmic Potentials on $\C^n$}

\subsection{ The euclidean logarithmic potential in $\C^n$}
Our main motivation is  to study projective logarithmic potentials on $\mathbb{P}^n$. It turns out that when localizing these potentials in affine coordinates, we end up with a kind of projective logarithmic potential on $\C^n$ which we shall study here. 

We first introduce the normalized  logarithmic kernel on $\C^n \times \C^n$ defined by
 $$
 K (z,w) := \frac{1}{2} \log\frac{\mid z-w\mid^{2}}{1+\mid w\mid^{2}}, \, \, \, (z,w) \in \C^n \times \C^n.
 $$
Let $\mu$ be a probability measure on $\C^n$. We define its logarithmic potential  as follows, for $z\in\mathbb{C}^{n}$ \begin{equation} \label{eq:potential}
 U_{\mu}(z) = \int_{\C^n} K (z,w) d \mu (w)
 =\frac{1}{2}\int_{\mathbb{C}^{n}} \log\frac{\mid z-w\mid^{2}}{1+\mid w\mid^{2}} d\mu(w)
 \end{equation}
It is well known that for any $w \in \C^n$, the function $K_w := K (\cdot,w)$ is plurisubharmonic in $\C^n$, $K_w \in \mathcal L_+ (\C^n)$ and satisfies the complex Monge-Amp\`ere equation
 $$
 (dd^c K (\cdot,w))^n = \delta_w,
 $$
 in the sense of currents on $\C^n$, where $\delta_w$ is the unit Dirac mass at $w$.
\begin{theo}  \label{thm:logpot}
 Let $\mu$ be a probability measure on $\C^n$. Then  for any $z \in \mathbb{C}^n$,
 $$
 U_{\mu}(z) \leq  \frac{1}{2} \log (1 + \vert z\vert^2).
 $$
 and if $\sigma_{2 n - 1}$ is the normalized Lebesgue measure on the unit sphere $\mathbb S^{2 n - 1}$,

 $$
 \int_{\{\vert z \vert = 1\}}  U_{\mu}(z)  d \sigma_{2 n - 1} (z) \geq -  \log \sqrt{5}.
 $$
 In particular $U_\mu \in \mathcal L (\C^n)$.

 If moreover $\mu$ safisfies the following logarithmic moment condition at infinity i.e.
 $$
 \int_{\C^n} \log (1 + \vert w \vert^2)  d \mu (w) < + \infty,
 $$
 then $U_{\mu} \in \mathcal L _\star (\C^n)$ and its Robin function $\rho_\mu := \rho_{U_{\mu}}$ satisfies the lower bound
 $$
  \int_{\mathbb{P}^{n - 1}}\rho_\mu (\zeta) \, \omega_{n - 1}  \geq - (1 \slash 2) \int_{\C^n} \log (1 + \vert w \vert^2)  d \mu (w),
 $$
 where $\omega_{n - 1}$ is the Fubini-Study volume form on $\mathbb{P}^{n - 1}$.
\end{theo}
\begin{proof}
 By (\ref{eq:Lagrange2}),  for any  $z, w \in \C^n$,

 $$
 \frac{\mid z-w\mid^{2}}{(1 + \vert z\vert^2) (1+\mid w\mid^{2})} \leq 1 - \frac{\vert 1 + z \cdot \bar {w} \vert^2}{(1 + \vert z\vert^2) (1+\mid w\mid^{2})} .
 $$

 Then for $z, w \in \C^n$,
 $$
 K (z,w) :=\frac{1}{2} \log\frac{\mid z-w\mid^{2}}{1+\mid w\mid^{2}}  \leq  \frac{1}{2} \log (1 + \vert z\vert^2).
 $$

 For any fixed $w \in \C^n$, the function  $z \longrightarrow K (z,w)$ is plurisubharmonic in $\C^n$.
 It follows that  $U_\mu \in \mathcal L (\C^n)$ provided that  we prove that $U_\mu \not \equiv - \infty$.

 To prove the last statement,   write for $z \in \C^n$
 $$
2 \,  U_{\mu}(z) = I_1 (z) + I_2 (z),
 $$
 where
 $$
 I_1 (z) := \int_{\vert w\vert \leq 2}  \log\frac{\mid z-w\mid^{2}}{1+\mid w\mid^{2}} d\mu(w),
 $$
 and

 $$
 I_2 (z) := \int_{\vert w\vert >  2}  \log\frac{\mid z-w\mid^{2}}{1 + \mid w\mid^{2}}   d\mu(w).
 $$
Let us prove that  $I_1 \not \equiv - \infty$ and $I_2 \not \equiv - \infty$. For the first integral observe that for $\vert z\vert \geq 3,$ we have
$$
 I_1 (z) \geq   -  \mu (\B (0,2)) \log  5 > - \infty.
 $$
 For the second integral, observe that for $\vert z \vert \leq  1$,  $\frac{(\vert w\vert - 1)^2}{1 + \mid w\mid^{2}} \geq \frac{1}{5}$ for $\vert w \vert \geq 2$. Hence
 for $\vert z \vert < 1$,
 $$
  I_2 (z) \geq  \int_{\vert w\vert >  2}  \log\frac{(\vert w\vert - 1)^2}{1 + \mid w\mid^{2}}  d \mu (w) \geq - \log 5 > - \infty.
 $$
Therefore $I_1$ and $I_2$ belongs to $PSH (\C^n) \subset L^1_{loc} (\C^n)$ and then $U_{\mu} \in PSH (\C^n)$.
 This proves that $U_{\mu}  \in \mathcal L (\C^n).$\\
Then for $z, w \in \C^n$,
 $$
 \log\frac{\mid z-w\mid^{2}}{1+\mid w\mid^{2}}  \geq  \log\frac{\mid \vert z \vert - \vert w\vert \mid^{2}}{1+\mid w\mid^{2}}.
 $$
 Fix $0 < \delta< 1$ and $z \in \C^n$ such that $ \vert z\vert = 1$. Then splitting the integral defining $U_\mu$ in two parts, we get  :
 $$
2  U_{\mu} (z) \geq   \log\frac{ (1- \delta)^2}{ 1 + \delta^{2}}  \mu (\B (0,\delta))
  + \int_{\vert w \vert \geq \delta}  \log\frac{\vert  z - w\vert^2}{ 1 + \mid w\mid^{2} }  d \mu (w)
 $$
 Hence by submean value inequality, we obtain
 \begin{eqnarray*}
2 \,  \int_{\vert z \vert = 1 }  U_{\mu} (\xi) d \sigma_{2 n - 1} (\xi)
 & \geq  & \log\frac{ (1- \delta)^2}{ 1 + \delta^{2}}  \mu (\B (0,\delta))  +  \int_{\vert w \vert \geq  \delta}  \log\frac{\vert  w\vert^2}{ 1 + \mid w\mid^{2}}  d \mu (w) \\
 & \geq &  \log\frac{ (1- \delta)^2}{ 1 + \delta^{2}} \mu (\B (0,\delta))  +  \mu (\{\vert w \vert \geq \delta\})  \log\frac{\delta^2}{1 + \delta^{2}}
 \end{eqnarray*}
Then taking $\delta= 1 \slash 2$ we obtain
$$
\int_{\vert z\vert = 1 }  U_{\mu} (z) d \sigma_{2 n - 1} (z)  \geq  - \log \sqrt{5}.
$$

This also implies that  $ U_{\mu} \not \equiv - \infty$.

 To prove the last statement observe that $\rho_\mu \leq 0$ on $\C^n$ and
 $$
 \lim_{\vert \lambda \vert \to + \infty} (\log \vert \lambda z-w\vert^{2} - \log \vert \lambda z\vert^{2}) = 0.
 $$
 Then apply Fatou's lemma to obtain the conclusion. 
\end{proof}

\subsection{ Riesz potentials}
Here we prove a technical lemma which is certainly well known but since we cannot find the right reference for it, we will give all the details needed in the sequel. Here we work in the euclidean space $\R^N$ with its usual scalar product and its associated euclidean norm $\Vert \cdot \Vert$.

Let $\mu$ be a probability measure with compact support on $\R^{N}$. Define its Riesz potentials  by
$$
 J_{\mu, \alpha}(x) :=   \int_{\mathbb{C}^{n}} \frac{d\mu(w)}{|x-y|^{\alpha}} = \mu \star J_\alpha (x), \, x \in \R^N,
 $$
 where
 $$
 J_\alpha (x) := \frac{1}{\vert x\vert ^\alpha}, \, \, x \in \R^N.
 $$
 Observe that  $J_\alpha \in L^p_{loc} (\R^N)$ if and only if $0 < p  < N \slash \alpha$.

 Given a probability measure $\mu$ on $\R^N$, there are many different notions of dimension for the measure $\mu$. Here we use the following one (see  \cite{LMW02}). We define the {\it L\'evy concentration} functions of $\mu$ as follows:
 $$
\mu (x,r) := \mu (B (x,r)),  \, \, Q_\mu (r) := \sup \{ \mu (x,r) ; x \in \mrm{Supp} \mu \},
 $$
where $B (x,r)$ is the euclidean (open)  ball of center $x$ and radius $r > 0$.

   The  lower {\it  concentration dimension} of $\mu$ is given by the following formula:
\begin{equation} \label{eq:dim}
 \gamma(\mu) = \gamma_-(\mu):= \liminf_{r \to 0^+} \frac{\log Q_\mu (r)}{\log r}\cdot
\end{equation}

 We will call it for convenience the dimension of the measure $\mu$. The following property is well known.
 
\begin{lem}\label{lem:dimension} Let $\mu$ be a probability measure $\mu$ with compact support on $\C^n$. Then  its dimension $\gamma(\mu)$ is  the supremum of all the exponents $\gamma\geq  0$ for which the following estimates are staisfied: there exists $C > 0$ such that $\forall x \in \R^N, \, \, \forall r \in ] 0,1[$, 
 $$
 \mu (B (x,r)) \leq C r^{\gamma}. 
 $$
Moreover we have $0 \leq \gamma(\mu) \leq N$.

\end{lem} 

\begin{proof} The first statement is obvious and we have $\gamma(\mu) \geq 0$ since $\mu$ has a finite mass. To prove the upper bound, observe that for any fixed $r > 0$, the function $x \longmapsto \mu (x,r)$ is a non negative Borel function on $\R^n$. By Fubini's Theorem, we have for any $r > 0$,
\begin{eqnarray*} \label{eq:concentration}
\int_{\R^N} \mu (x,r) d x  & = & \int_{\R^N} \left(\int_{y \in B (x,r)} d \mu (y)\right) d x \nonumber \\
& = & \int_{\R^N} \left(\int_{B (y,r)}  d x \right) d \mu (y).
\end{eqnarray*}
If we denote by $\tau_N$ the volume of the euclidean unit ball in $\R^N$, we obtain
\begin{eqnarray} \label{eq:concentration}
\int_{\R^N} \mu (x,r) d x  =  \tau_{N} r^{N}.
\end{eqnarray}

Therefore for any $r > 0$, we have 
$$
\tau_{N} r^{N} \leq Q_\mu (r) \, \mu (\R^N)  =  Q_\mu (r) ,
$$
which implies immediately that $\gamma(\mu) \leq N$.
\end{proof}

\begin{exa}
1. Then the measure $\mu$ has no atom in $\R^N$  if and only if $\gamma (\mu) > 0$.

 2. Let $0 < k \leq N$ be any real number and let $A \subset \R^N$ be any  Borel subset such that its $k$-dimensional Hausdorff measure satisfies  $0 < \lambda_k (A) < + \infty$. Then the restricted measure $\lambda_{k, A} := {\bf 1}_A  \lambda_k$ is a Borel measure of dimension $k$. In particular
 the $N$-dimensional Lebesgue measure $\lambda_{k, A}$ has dimension $N$.
 
\end{exa} 
\begin{lem}\label{lem:Rieszpotential}
 Let $\mu$ be a probability measure with compact support on $\R^N$ and $0 < \alpha < N$ and  $\gamma:= \gamma(\mu)$ its dimension. Then

  $$
  J_{\mu, \alpha} \in L^{p}_{Loc}(\R^N)\quad \, \, \hbox{if} \quad \, \, 1  < p <\frac{N - \gamma}{(\alpha - \gamma)_+},
 $$
 where $x_+ := \max \{x,0\}$ for a real number $x \in \R$.
 \end{lem} 
\begin{proof}
We follow an idea from \cite{P16} where the case when $\alpha = N - 1$ is considered (see \cite[Proof of Lemma 10.12]{P16}). For convenience, we give all the details here.

By the Cavalieri principle for any fixed $x \in \R^N$,we have
$$
J_{\mu,\alpha} (x) = \alpha  \int_0^{+ \infty} \mu (x,r) \frac{d r}{r ^{\alpha + 1}}.
$$

Then by Minkowski inequality, we obtain
$$
\Vert  \mu \star J_\alpha \Vert_{p} \leq \alpha  \int_0^{+ \infty}  \Vert \mu (\cdot,r)\Vert_p \frac{d r}{r ^{\alpha + 1}},
$$
here $\Vert \cdot \Vert_p$ means the $L^p$-norm with respect to the Lebesgue measure on $\R^N$.

Recall that  $Q_\mu (r) := \sup_{x} \mu (x,r)$ for $r > 0$. This is a bounded Borel function on $\R^+$ such that $0 \leq Q_\mu (r) \leq 1$ since $\mu$ is a Probability measure. 
For $ p >  1$ fixed, we can write for any $x \in \R^N$ and $r > 0$,
\begin{eqnarray*}
 \mu (x,r)^p & = &  \mu (x,r)^{p- 1}  \mu (x,r) \leq Q_\mu (r)^{p - 1} \mu (x,r) \\
 & \leq & \min \{Q_\mu (r), 1\}^{p - 1} \mu (x,r).
\end{eqnarray*}

Then by (\ref{eq:concentration}) we get, 
\begin{eqnarray*}
\Vert \mu (\cdot,r)\Vert_p^p & = & \int_{\R^N} \mu (x,r)^p d x \\
& \leq &  \min \{Q_\mu (r), 1\}^{p - 1} \int_{\R^N} \mu (x,r) d x \\
& \leq & C \min \{Q_\mu (r), 1\}^{p - 1} r^{ 2 n}.
\end{eqnarray*}

Therefore for any fixed $\kappa > 0$, we have

\begin{eqnarray*}
\Vert  \mu \star J_\alpha \Vert_{p} & \leq &  C \alpha  \int_0^{\kappa}  Q_\mu (r)^{(p - 1) \slash p} \, r^{ 2 n \slash p}\frac{d r}{r ^{\alpha + 1}} \\
& + & \int_{\kappa}^{\infty}  r^{ 2 n \slash p}\frac{d r}{r ^{\alpha + 1}}.
\end{eqnarray*}

Using the estimate on $Q_\mu$ in terms of the dimension and minimizing in $\kappa > 0$, we get the required result.
 \end{proof}

%\section{Projective Logarithmic Potentials in $\mathbb C^n$} 

\subsection{The projective logarithmic kernel}
Our main motivation is  to study projective logarithmic potentials on $\mathbb{P}^n$. It turns out that when localizing these potentials in affine coordinates, we end up with a kind of projective logarithmic potential on $\C^n$ which we shall study first.

 Recall the definition of the logarithmic kernel from the previous section,
 $$
 K (z,w) := \frac{1}{2} \log\frac{\mid z-w\mid^{2}}{1+\mid w\mid^{2}},
 $$
 The projective logarithmic kernel on $\C^n \times \C^n$ is defined by the following formula:
 $$
 N (z,w) := \frac{1}{2} \log\frac{\mid z-w\mid^{2} + \vert z \wedge w \vert^2}{1+\mid w\mid^{2}} , \, \, (z,w) \in \C^n \times \C^n.
 $$
 \begin{lem} \label{lem:logKern}  1. The kernel $N$ is  upper semi-continuous in $\C^n \times \C^n$ and smooth off the diagonal of  $\C^n \times \C^n$.

 2. For any fixed $w \in \C^n$, the function  $N_w := N (\cdot,w) : z \longmapsto N (z,w)$ is plurisubharmonic in $\C^n$ and satisfies the following inequality
 $$
 K (z,w) \leq N (z,w) \leq (1 \slash 2) \log (1 + \vert z\vert^2), \, \, \forall (z,w) \in \C^n \times \C^n
 $$
 hence for any $w \in \C^n$,  $N (\cdot,w) \in \mathcal L_+ (\C^n)$.

 3. The kernel $N$  has a logarithmic singularity along the diagonal i.e. for any $(z,w) \in \C^n \times \C^n$,
 $$
 0 \leq  N (z,w)  -  K (z,w) \leq  \frac{1}{2} \log (1 + \min \{\vert z\vert ^2, \vert w\vert ^2\}).
 $$

4. For any $w \in \C^n$,
 $$
 (dd^c N_w)^n = \delta_w,
 $$
 where $\delta_w$ is the unit Dirac measure on $\C^n$  at the point $w$.
\end{lem}
\begin{proof}  The first  statement is obvious. Let us prove the second  one. By the formula \ref{eq:Lagrange2} we have   for any  $z, w \in \C^n$,

 $$
 \frac{\mid z-w\mid^{2} + \vert z \wedge w\vert^2}{(1 + \vert z\vert^2) (1+\mid w\mid^{2})} = 1 - \frac{\vert 1 + z \cdot \bar {w} \vert^2}{(1 + \vert z\vert^2) (1+\mid w\mid^{2})}.
 $$
 Thus for $z, w \in \C^n$,
 $$
 N(z,w) = \frac{1}{2} \log\frac{\mid z-w\mid^{2} + \vert z \wedge w\vert^2}{1+\mid w\mid^{2}}  \leq \frac{1}{2} \log (1 + \vert z\vert^2).
 $$
This yields
$$
  N (z,w)  -  K (z,w) =  (1 \slash 2) \log \left(1 + \frac{\mid z\wedge w\mid^{2} }{\vert z - w\vert^2}\right).
$$
Now observe that by  Lemma \ref{eq:Lagrange2}, we have
$$
\vert z \wedge w\vert^2  = \vert (z - w)\wedge w\vert^2  \leq  \vert z - w \vert^2 \vert w\vert^2.
$$
By symmetry we obtain the required inequality.

To prove the third property, observe that for a fixed $w \in \C^n$,
the two functions
$$
u (z) := N (z,w), \, \, v (z) := K (z,w),
$$
belongs to $\mathcal L_+ (\C^n)$. 
Hence by (\ref{eq:totalmass}) they have the same total Monge-Amp\`ere mass in  $\C^n$ i.e.
$$
\int_{\C^n} (dd^c u)^n = 1.
$$
On the other we know that
$$
(dd^c v)^n = \delta_w.
$$
 Since $\lim_{z \to w} \frac{u (z)}{v (z)} = 1$, by the comparison theorem of Demailly \cite{De93}, they have the same residual Monge-Amp\`ere mass at the point $w$. Therefore the total Monge-Amp\`ere mass of the measure $ (dd^c u)^n$  is concentrated at the point $w$, which proves our statement.
\end{proof}

\subsection{The projective logarithmic potential}
Let $\mu$ be a probability measure with compact support on $\C^n$. We define the projective logarithmic potential of $\mu$ as follows  :
\begin{equation} \label{eq:potential}
 V_{\mu}(z)=\frac{1}{2}\int_{\mathbb{C}^{n}} \log \left(\frac{\mid z-w\mid^{2} + \vert z \wedge w \vert^2}{1+\mid w\mid^{2}}\right) d\mu(w),
 \end{equation}
It follows that  $V_\mu \in \mathcal L (\C^n)$ provided that  we prove that $V_\mu \not \equiv - \infty$.
\begin{theo}  \label{lem:logprojec}
 Let $\mu$ be a probability measure with compact support on $\C^n$. Then  for any $z \in \mathbb{P}^n$,
 $$
 U_\mu (z) \leq V_{\mu}(z) \leq  \frac{1}{2} \log (1 + \vert z\vert^2).
 $$
 and
$$
 \int_{\{\vert z \vert = 1\}}  V_{\mu}(z)  d \sigma_{2 n - 1} (z) \geq - \log \sqrt{5}.
 $$
 Moreover $V_{\mu} \in \mathcal L _\star (\C^n)$ and its Robin function $\rho_\mu := \rho_{V_{\mu}} = \rho_{U_{\mu}}$ satisfies the lower bound
 $$
  \int_{\mathbb P^{n - 1} }\rho_\mu (\xi) \, \omega_{n - 1}   \geq - (1 \slash 2) \int_{\C^n} \log (1 + \vert w \vert^2)  d \mu (w),
 $$
 where $\omega_{n - 1}$ is the Fubini-Study volume form on $\mathbb P^{n - 1}.$
 \end{theo}
\begin{proof} The right hand side estimate in the first inequality follows from lemma 3.5, while the other statements follow from Theorem \ref{thm:logpot} since $V_\mu \geq U_\mu$ in $\C^n$.
\end{proof}

\section{The projective logarithmic potential on $\Pn$ }

 \subsection{The projective logarithmic kernel}
The projective logarithmic kernel on $\mathbb P^n \times \mathbb P^n$ is defined by the following formula :
 $$
  G (\zeta,\eta) :=  (1 \slash 2) \log \, \frac{\vert \zeta \wedge \eta \vert^2}{\vert \zeta \vert^{2} \vert \eta \vert^2} , \, \, (\zeta,\eta) \in \mathbb P^n \times \mathbb P^n.
 $$
 \begin{lem} \label{lem :logKern}  1. The kernel $G$ is a non positive upper semi-continuous function in $\mathbb P^n \times \mathbb P^n$ and smooth off the diagonal of  $\mathbb P^n \times \mathbb P^n$.

 2. For any fixed $\eta \in \Pn $, the function  $G  (\cdot,\eta) : \zeta \longmapsto G (\zeta,\eta)$ is a non positive $\omega$-plurisubharmonic in $\Pn$ which is smooth in $\Pn \setminus \{\eta\}$  i.e.   $G (\cdot,\eta) \in PSH (\Pn, \omega)$.

 3. The kernel $G$  has a logarithmic singularity along the diagonal. More precisely   for a fixed  $\eta \in \mathcal U_k$, in the coordinate chart  $(\mathcal U_k, {\bf z}^k)$ we have
 $$
 0 \leq  G (\zeta,\eta)  -  (1 \slash 2) \log\frac{\mid {\bf z}^k (\zeta) -{\bf z}^k (\eta) \mid^{2} }{  (1+\vert {\bf z}^k (\zeta)\vert^2)  (1+\vert {\bf z}^k (\eta))\vert^{2}} \leq  \log (1 + \min \{\vert {\bf z}^k (\zeta)\vert ^2, \vert {\bf z}^k (\eta)\vert ^2\}).
 $$

4. For any $\eta \in \mathbb P^n  $, $G (\cdot,\eta) \in DMA_{loc} (\mathbb P^n,\omega)$ and
 $$
 (\omega + dd^c G (\cdot,\eta))^n = \delta_\eta,
 $$
 where $\delta_\eta$ is the unit Dirac measure on $\mathbb P^n $  at the point $\eta$.
\end{lem}
 \begin{proof}
 This lemma  follows from Lemma \ref{lem:logKern} by observing that in each open chart $(\mathcal U_k, {\bf z}^k)$
 we have for $(\zeta,\eta) \in \mathcal U_k \times \mathcal U_k$,
 $$
  G (\zeta,\eta)  = N (z,w) - \frac{1}{2} \log (1 + \vert z\vert^2) ,
 $$
 where $z := {\bf z}^k (\zeta) $ and $w := {\bf z}^k (\eta)$.
\end{proof}

\subsection{The projective logarithmic potential}
Let $\mrm{Prob} (\mathbb P^n )$ be the convex compact set of  probability measures on $\mathbb P^n $. Given $\mu \in \mrm{Prob} (\mathbb P^n )$ we define its (projective) logarithmic potential as follows
\begin{eqnarray*}
\G_\mu (\zeta) :&=& \int_{\mathbb P^n } G (\zeta,\eta) d \mu (\eta), \\
&=& \int_{\mathbb P^n }\log \frac{\vert \zeta \wedge \eta \vert}{\vert \zeta \vert \vert \eta \vert}\ d \mu (\eta),\\
\end{eqnarray*}
As observed in (\cite{As17}) the projective kernel $\G$ can be expressed in terms of the geodesic distance $d$  on the K\"ahler maniflod $(\mathbb P^n,\omega_{FS})$. Namely we have
$$
\G_\mu (\zeta) =\int_{\mathbb P^n }\log\sin\bigl({d(\xi,\eta)\over\sqrt{2}}\bigr)d\mu(\eta).
$$

  Thanks to this formula , we see that if  $f$ is a radial function on $\mathbb P^n$ i.e. $ f(\zeta) := g (d (\zeta,a))$  for a fixed point $a \in \mathbb P^n$. Then choosing polar coordinates around $a$ we have
  \begin{equation} \label{eq:polarcoord}
  \int_{\mathbb{P}^{n}} f (\zeta) d V (\zeta) )=\int_{0}^{\pi/\sqrt{2}}g(r)A(r)dr,
   \end{equation}
where $A(r)$ is the "area" of the sphere of center $a$ and radius $r$. The expression of $A(r)$ is given by the formula
$$
A(r)=c_{n}\sin^{2n-2}(r/\sqrt{2})\sin(\sqrt{2}r),
$$
where $c_{n}$ a constant depending on the volume of the unit ball in $\mathbb{R}^{2n}$. For more details, see \cite[Page 168]{Rag71},\cite[Section 3]{AB77}  or \cite[Lemma 5.6]{Hel65}.
% Now, we will present a result of D.L.Ragozin that how to compute integrals involving the Riemannian volume form $dV$ of $h$. For a proof see \cite[Lemma 4.4]{Rag71}.

This allows to give a simple example.
\begin{pro}
1. Let $\sigma$ be the Lebesgue measure asociated to the Fubini-Study volume form $dV$. Then for any $\zeta \in \mathbb P^n$,
$$
G_\sigma (\zeta) = - \alpha_n,
$$
where $\alpha_n > 0$ is a numerical constant given by the formula (\ref{eq:numconstant}) below. 

2. For any $\mu \in \text{Prob} (\mathbb P^n)$, $\G_\mu$ is a negative $\omega$-plurisubharmonic function in $\mathbb P^n$ such that
\begin{equation} \label{eq:normalization}
\int_{\mathbb P^n} \G_\mu (\zeta) d V (\zeta) =  - \alpha_n.
\end{equation}
\end{pro} 

\begin{proof} We use the same computations based on the formula (\ref{eq:polarcoord})  as in (\cite{As17}).

1.  By (\ref{eq:polarcoord}) and the co-area formula
\begin{eqnarray*}
 \G_{\sigma}(\zeta) &=&\int_{\mathbb P^n}\log\sin\bigl({d(\zeta,\eta)\over\sqrt{2}}\bigr)dV(\eta) \nonumber \\
 &=&\int_{0}^{\pi\over\sqrt{2}}\log\sin\bigl({r\over\sqrt{2}}\bigr)A(r)dr \nonumber \\
 &=& c_{n}\int_{0}^{\pi\over\sqrt{2}}\log\Bigl(\sin\bigl({r\over\sqrt{2}}\bigr)\Bigr)\sin^{2n-2}\bigl({r\over\sqrt{2}}\bigr)
 \sin(\sqrt{2}r)dr \nonumber \\
 &=& 2\sqrt{2}c_{n}\int_{0}^{\pi\over 2}(\log\sin(t))\sin^{2n-1}(t)\cos(t)dt\\
 &=&2\sqrt{2}\int_{0}^{1}u^{2n-1}\log u du.
 \end{eqnarray*}
 Therefore the statement follows if we set
 \begin{equation} \label{eq:numconstant}
  \alpha_n := \sqrt{2}\int_{0}^{1}u^{2n-1}\log u du
 \end{equation}

2. To prove the second statement it is enough to observe the following symetry 
$$
 \int_{\mathbb P^n}  \G_\mu (\zeta) d \sigma (\zeta)  = \int_{\mathbb P^n} \G_\sigma (\eta)  d \mu (\eta)  = - \alpha_n \mu (\mathbb P^n) = - \alpha_n,
$$
thanks to the formula (\ref{eq:normalization}).
\end{proof}
 
 \section{The Monge-Amp\`ere measure of the  potentials}

\subsection{The  Monge-Amp\`ere measure of $\mathbb{V}_{\mu}$ in $\C^n$}
 We begin  this section by showing that $V_{\mu}$   belongs to the domain of definition of the complex Monge-Amp\`ere operator for $(n\geq 3)$.
\begin{theo}
Let $\mu$ be a probability  measure  on  $\mathbb{C}^{n}$ $ (n\geq 2)$ with compact support. Then

1)   $V_{\mu}\in DMA_{loc}(\mathbb{C}^{n})$.

2) $V_\mu \in W_{loc}^{2,p} (\C^n)$ for any $0 < p < n$. In particular, for any $1 \leq k \leq n - 1$, the measure $(dd^c V_\mu)^k \wedge \omega^{n - k}$ is absolutely continuous with respect to the Lebesgue measure on $\C^n$.
 \end{theo}

 \begin{proof} The first part of the theorem was proved in \cite{As17} but for convenience we reproduce the proof here since we will use the same compuations to prove the other statements.
 
1. We will use the characterization due to B\l ocki (\cite{Bl06}). For  $\e > 0$, set  
$$
V_{\mu}^{\e}(z):=\frac{1}{2}\int_{\mathbb{C}^{n}}\log\Big(\frac{|z-w|^{2}+|z\wedge w|^{2} +\e^{2}}{1+\mid w\mid^{2}}\Big) d\mu(w).
$$
 It is clear that $V_{\mu}^{\e} \in \mathcal L (\C^n) \cap C^{\infty}(\mathbb{C}^{n})$ and decreases towards $V_{\mu}$ as $\e$ decreases to $0$.
 Since $V_\mu$ in plurisubharmonic in $\C^n$, we have
$$
|V_{\mu}^{\e}|^{n-p-2}
\in L^{r_{1}}_{Loc}(\mathbb{C}^{n})\quad\hbox{ for any}\quad \ r_{1} > 0.
$$
On the other hand, recall that  
$$
|z\wedge w|^{2}=\sum_{1 \leq i < k \leq n} |z_i w_k - z_k w_i|^{2},
$$
and observe that
$$
  \frac{\partial}{\partial z_{m}} (|z-w|^{2}+|z\wedge w|^{2}) = \overline{z_m - w_m} + \sum_{m < j \leq n} w_j (\overline{z_m w_j - z_j w_m}) - \sum_{1 \leq i < m} w_i (\overline{z_i w_m - z_m w_i}).
$$
Then we have for any $z \in \C^n$, 
\begin{eqnarray*}
  & & 2 \frac{\partial}{\partial z_{m}} V_{\mu}^{\e}(z) \\
 &=& \int_{\mathbb{C}^{n}} \frac{\overline{z_m - w_m} + \sum_{m < j \leq n} w_j (\overline{z_m w_j - z_j w_m}) - \sum_{1 \leq i < m} w_i (\overline{z_i w_m - z_m w_i})}{|z-w|^{2}+|z\wedge w|^{2}+ \e^{2}} d\mu(w).
 \end{eqnarray*}
 Thus
 
 \begin{eqnarray*}  
 | \nabla V_{\mu}^{\e}(z)| 
 &\leq & \frac{1}{2}\int_{\mathbb{C}^{n}}\frac{|z-w|+|w||z\wedge w|}{|z-w|^{2}+|z\wedge w|^{2}}d\mu(w)  \\
   &\leq & \frac{\sqrt{2}}{2}\int_{\mathbb{C}^{n}}\frac{(1+|w|)\sqrt{|z-w|^{2}+|z\wedge w|^{2}}}{|z-w|^{2}+|z\wedge w|^{2}}d\mu(w) \nonumber \\
   &\leq & \frac{\sqrt{2}}{2}\int_{\mathbb{C}^{n}}\frac{(1+|w|)}{|z-w|}d\mu(w) \nonumber \\
   &\leq & \frac{\sqrt{2}}{2} + \frac{\sqrt{2}}{2} (1+\vert z\vert)\int_{\mathbb{C}^{n}}\frac{d\mu(w)}{|z-w|}.    
  \end{eqnarray*}

In conclusion we have for $z \in \C^n$,
\begin{equation} \label{eq:gradupperbound}
 \Vert \nabla V_{\mu}^\e (z) \Vert \leq \frac{\sqrt{2}}{2} + \frac{\sqrt{2}}{2} (1 + \vert z\vert) J_{\mu,1} (z).
\end{equation}
By Lemma \ref{lem:Rieszpotential} we conclude that   
$$
   |\nabla V_{\mu}^{\e}|\in L^{q}_{Loc}(\mathbb{C}^{n}), \forall q<2n,
    $$
  and so   
  $$ |\nabla V_{\mu}^{\e}(z) |^{2} \in L^{r_{2}}_{Loc}(\mathbb{C}^{n})\ for\ r_{2}<n$$

The same computation shows that the second partial derivatives satisfy
\begin{eqnarray} \label{eq:2ndOrder}
\left\vert \frac{\partial^{2}}{\partial\overline{z_{k}}\partial z_{m}} V_{\mu}^{\e}(z)\right\vert + \vert \frac{\partial^{2}}{\partial{z_{k}}\partial z_{m}} V_{\mu}^{\e} (z)\vert & \leq & c_n \int_{\mathbb{C}^{n}}\frac{(1+\vert w\vert^2)}{|z-w|^{2}+|z\wedge w|^{2}}d\mu(w) \nonumber \\
& \leq & c_n + c_n (1 + \vert z\vert^2) J_{\mu,2} (z).
\end{eqnarray}
Hence recalling that $\mu$ has compact support and using  Lemma \ref{lem:Rieszpotential}, we obtain the inequality
\begin{eqnarray} \label{eq:secondorderest}
   \Big|\nabla^2  V_{\mu}^{\e}(z)\Big| & \leq & c_n + C (n,\mu)  J_{\mu,2} \in L^{p}_{loc}(\mathbb{C}^{n}), 
  \end{eqnarray} 
  for any $ p <n,$ which proves the second statement.

To prove the first statement, observe that  for a fixed compact set $K \subset \C^n$, we have
\begin{eqnarray*}
 &{}&\int_{K}| {V}_{\mu}^{\e}|^{n-p-2} d {V}_{\mu}^{\e}\wedge d^{c} V_{\mu}^{\e}\wedge(dd^{c} V_{\mu}^{\e})^{p}\wedge \omega^{n-p-1}\\ 
 &\leq & \int_{K}|(V_{\mu}^{\e})^{n-p-2}|| \nabla V_{\mu}^{\e} |^{2}\wedge \underbrace{dd^{c} V_{\mu}^{\e}\wedge ...\wedge dd^{c} V_{\mu}^{\e}}_{p-times} \wedge \omega^{n-p-1}\\
 &\leq & C \displaystyle\sum_{l,k=1}^{n}\int_{K}|(V_{\mu}^{\e})^{n-p-2}|| \nabla V_{\mu}^{\e} |^{2}  \Big|\frac{\partial^{2}}{\partial z_{k_{1}}\partial\overline{z}_{l_{1}}} V_{\mu}^{\e}(z)\Big|... \Big|\frac{\partial^{2}}{\partial z_{k_{p}}\partial\overline{z}_{l_{p}}} V_{\mu}^{\e}(z)\Big|\omega^{n}\\
  &\leq & C \displaystyle\sum_{l,k=1}^{n} \|(V_{\mu}^{\e})^{n-p-2}\|_{r_{1}}\| \nabla V_{\mu}^{\e} \|^{2}_{r_{2}}
   \Big\|\frac{\partial^{2}}{\partial z_{k_{1}}\partial\overline{z}_{l_{1}}} V_{\mu}^{\e}(z)\Big\|_{s_{1}}... \Big\|\frac{\partial^{2}}{\partial z_{k_{p}}\partial\overline{z}_{l_{p}}} V_{\mu}^{\e}(z)\Big\|_{s_{p}},
 \end{eqnarray*}
 where $\Vert \cdot \Vert_{t}$ denotes the norm in $L^t (K)$. 

Since $0\leq p\leq n-2$, in this estimates we have a product of $p + 2 \leq n$ terms such that $p + 1 \leq  n- 1$ terms are in $L^k_{loc}$ with $k < n$ and one term in $L^k_{loc}$ for any $k > 0$. In order to apply H\"older inequality, we need to choose $ r_1 > 1$ and $r_2, s_1, \cdots, s_n \in ]1,n[$ such that 
  $$
  {1\over r_{1}}+{1\over r_{2}}+ {1\over s_{1}}+...+{1\over s_{p}}=1
  $$
Indeed since $p + 1 <  n$, we can set  $r_{2}=s_{j}=p+1+\epsilon<n-\epsilon$  for $j=1,...,p$ and $r_{1}=\frac{p+1+\epsilon}{\epsilon}$ to obtain the required condition.
Thus the complex Monge-Amp\`ere measure $(dd^{c} V_{\mu})^{n}$ is well defined by B\l ocki's Theorem \ref{thm:Bl}.

2. The computation above shows that for $1 \leq k, m \leq n$
\begin{equation}
  \frac{\partial^{2}}{ \partial z_{k} \partial z_{m}} V_{\mu}^\e (z) \leq C J_{\mu,2} \in L^{p}_{loc}(\mathbb{C}^{n}),
\end{equation}
for any $0 < p < n$ with a uniform constant $C > 0$.

Since $\frac{\partial^{2}}{ \partial z_{k} \partial z_{m}} V_{\mu}^\e  \to \frac{\partial^{2}}{ \partial z_{k} \partial z_{m}} V_{\mu}$ weakly on $\C^n$ as $\e \to 0$, it follows from standard Sobolev space theory that $\frac{\partial^{2}}{ \partial z_{k} \partial z_{m}} V_{\mu} \in L^p $ for any $0 < p < n$.

\end{proof}

We note for later use that the kernel
  $$
 N (z,w) := \frac{1}{2} \log\frac{\mid z-w\mid^{2} + \vert z \wedge w \vert^2}{1+\mid w\mid^{2}} , \, \, (z,w) \in \C^n \times \C^n.
 $$ 
 can be approximated by the smooth kernels
  $$
 N_{\epsilon} (z,w) := \frac{1}{2} \log\Big(\frac{\mid z-w\mid^{2} + \vert z \wedge w \vert^2 +\epsilon^{2}}{1+\mid w\mid^{2}}\Big) , \, \, (z,w) \in \C^n \times \C^n.
 $$

Since each function $N_\e (\cdot,w_j) \in \mathcal L_+ (\C^n)$, we know by  Lemma \ref{lem:Rieszpotential}, that the measures
$$
dd^{c}_{z}N_\e (z,w_{1})\wedge...\wedge dd^{c}_{z}N_\e (z,w_{n})
$$
are well defined probability measures on $\C^n$. Hence  we have 
for all  $w_{1},...,w_{n}\in \mathbb{C}^{n}$
$$
\int_{\mathbb{C}^{n}}dd_{z}^{c}N_{\epsilon}(z,w_{1})\wedge...\wedge dd_{z}^{c}N_{\epsilon}(z,w_{n}) = 1.
$$
\begin{pro}
   If $\mu$ be a  probability measure on  $\mathbb{C}^{n}$  then the complex Hessian currents associated to $V_{\mu}$ are given by the following formula: for any $1 \leq k \leq n$,
  $$
(dd^{c} V_{\mu})^{k}=\int_{(\mathbb{C}^{n})^k}dd^{c} N(.,w_{1})\wedge...\wedge dd^{c}N(.,w_{k})d\mu(w_{1})...d\mu(w_{k}),
   $$
   in the sense of $(k,k)$-currents on $\C^n$.
\end{pro}
 \begin{proof}
Since $V_{\mu}$ belongs to the domain of definition $DMA_{loc}(\mathbb{C}^{n})$, 
 the complex Monge-Amp\`ere current  $(dd^{c} V_{\mu})^{k}$ is well defined. 
 
 We assume that $2 \leq k \leq n$.  Set
  \begin{eqnarray*}
 V_{\mu,\epsilon}(z)&:=&{1\over 2}\int_{\mathbb{C}^{n}}\frac{1}{2} \log\Big(\frac{\mid z-w\mid^{2} + \vert z \wedge w \vert^2 +\epsilon^{2}}{1+\mid w\mid^{2}}\Big)d\mu(w)\\
&=& \int_{\mathbb{C}^{n}}N_{\epsilon}(z,w) d\mu(w)
\end{eqnarray*}

Let $\chi\geq 0$ be a positive smooth $(n-k,n-k)$-form with compact support in  $\mathbb{C}^{n}$ and denote for $j = 2, \cdots, n$, by $ \mu^{(j)} = \mu^{\otimes j}$ the product measure on $(\C^n)^j$. Then applying Fubini's theorem and integration by parts formula, we obtain
{\footnotesize
\begin{eqnarray*}
   A_\e  &:= & \int_{z\in \mathbb{C}^{n}}\chi(z) \wedge \int_{(\mathbb{C}^{n})^k} dd^{c} N_{\epsilon}(z,w_{1})\wedge...\wedge dd^c N_{\epsilon}(z,w_{k})d\mu^{(k)} (w)\\
   &=& \int_{(\mathbb{C}^{n})^k}\Big(\int_{z\in \mathbb{C}^{n}} N_{\epsilon}(z,w_{1}) dd^{c} \chi(z) \wedge dd^{c} N_{\epsilon}(z,w_{2})\wedge...\wedge dd^{c} N_{\epsilon}(z,w_{k})\Big)d\mu^{(k)} (w),
 \end{eqnarray*}
 }
where  $w := (w_1, \cdots,w_k) \in (\C^n)^k$.

Integrating by parts and  applying again Fubini's theorem  we obtain
{\footnotesize
\begin{eqnarray*}
&& A_\e \\
 & = &  \int_{z\in \mathbb{C}^{n}}\Big(\int_{\mathbb{C}^{n}}N_{\epsilon}(z,w_{1})d\mu(w_{1})\Big) \int_{(\mathbb{C}^n)^{k - 1}}dd^{c}_{z}\chi(z)\wedge \dots \wedge dd^{c}_{z}N_{\epsilon}(z,w_{k}) d\mu ^{(k- 1)} (w') \\
 &=& \int_{z\in \mathbb{C}^{n}} V_{\epsilon}(z) \int_{(\mathbb{C}^n)^{k- 1}}dd^{c} \chi(z) \wedge dd^{c} N_{\epsilon}(z,w_{2}\wedge \dots \wedge dd^{c} N_{\epsilon}(z,w_{n}) d\mu ^{(n- 1)} (w').
\end{eqnarray*}
}
where $w' := (w_{2}, \dots, w_{k}) \in \C^{k-1})$.

Using Fubini's theorem and integrating parts once again, we obtain when $k \geq 2$
{\footnotesize
\begin{eqnarray*}
&& A_\e   \\
& = &   \int_{(\mathbb{C}^n)^{k- 2}}\Big(\int_{z\in \mathbb{C}^{n}} N_{\epsilon}(z,w_{2}) dd^{c} \chi \wedge dd^{c} V_{\epsilon} \wedge dd^{c} N_{\epsilon}(z,w_{3} \wedge \cdots \wedge dd^{c}_{z}N_{\epsilon}(z,w_{k}) \Big) d\mu^{(k - 2)}.
\end{eqnarray*}
}
Repeating this process $k$ times we get the final equation

{\footnotesize
\begin{eqnarray} \label{eq:current}
       &{}&\int_{\mathbb{C}^{n}}\chi \wedge \int_{(\mathbb{C}^{n})^k}dd^c N_{\epsilon}(\cdot,w_{1})\wedge...\wedge dd^{c} N_{\epsilon}(\cdot,w_{k})d\mu^{(k)} (w)\\
     &=&\int_{\mathbb{C}^{n}}\chi \wedge (dd^{c} V_{\mu}^{\epsilon})^{k}. \nonumber 
   \end{eqnarray}
   }
Now  we want to pass to the limit as $\e \searrow 0$. 
The first term can be written as follows:

\begin{eqnarray*}
&& \int_{\mathbb{C}^{n}}\chi \wedge \int_{(\mathbb{C}^{n})^k}dd^c N_{\epsilon}(\cdot ,w_{1})\wedge...\wedge dd^{c} N_{\epsilon}( \cdot,w_{k})d\mu^{(k)} (w_{1}, \cdots, w_{k}) \\
 & = & \int_{(\C^n)^k} I_\e (w_1, \cdots, w_k) d \mu^{(k)} (w_1, \cdots, w_k),
\end{eqnarray*}
where
$$
I_\e (w_1, \cdots, w_k) := \int_{\mathbb{C}^{n}}\chi \wedge dd^{c}N_{\e}(\cdot ,w_{1}) \wedge \cdots \wedge dd^{c}N_{\e}(\cdot,w_{k}).
$$

 Observe that for any fixed $(w_1, \cdots, w_k) \in (\C^n)^k$ 
 $$
  dd^c N_{\epsilon}(\cdot,w_{1})\wedge...\wedge dd^{c}N_{\epsilon}(\cdot,w_{k})  \to
   dd^{c}N(\cdot,w_{1})\wedge...\wedge dd^{c} N(\cdot,w_{k}) 
   $$
   weakly in the sense of currents on $\mathbb{C}^n$ as $\e \searrow 0$. Hence the family of  functions $  I_\e (w_1, \cdots, w_k) $
 are uniformly bounded on $(\C^n)^n$ and by Fubinis'theorem, it converges as $\e \to 0$ pointwise to
the function
\begin{eqnarray*}
I (w_1, \cdots, w_k) &:= & \int_{\mathbb{C}^{n}}\chi \wedge dd^{c} N(.,w_{1})\wedge \cdots \wedge dd^{c} N(.,w_{k}).
\end{eqnarray*}
Therefore by Lebesgue convergence theorem we conclude that
\begin{eqnarray*}
&& \lim_{\e \to 0} \int_{(\C^n)^k} I_\e (w_1, \cdots, w_k) d \mu^{(k)} =  \int_{(\C^n)^k} I (w_1, \cdots, w_k) d \mu^{(k)}\\
& = & \int_{\mathbb{C}^{n}}\chi \wedge \int_{(\mathbb{C}^{n})^k}dd^c N (\cdot ,w_{1})\wedge...\wedge dd^{c} N (\cdot,w_{k})d\mu^{(k)}(w_1, \cdots, w_k).
\end{eqnarray*}

For the second term in (\ref{eq:current}), observe that 
since $V_{\mu,\e} \searrow V_\mu$ as $\e \searrow 0$, it follows by the convergence theorem that the second term converges also and
$$
\int_{\mathbb{C}^{n}}\chi \wedge (dd^{c} V_{\mu, \epsilon})^{k} \to  \int_{\mathbb{C}^{n}}\chi \wedge (dd^{c} V_{\mu})^{k}
$$
as $\e \searrow 0$.

Now passing to the limit in (\ref{eq:current}), we obtain the required statement.
\end{proof}

We will need a more general result. Let $\phi$ be a plurisubharmonic function in $\C^n$ such that $\phi \in DMA_{loc} (\C^n)$.
We define the twisted potential associated to a Probability measure $\mu$ by setting
$$
{V^{\phi}_\mu}  := V_\mu + \phi.
$$
Then we have the following representation formula:
$$
V_\mu^{\phi} (z) = \int_{\C^n} N^{\phi} (z,w) d \mu (w),
$$
where
$$
N^{\phi} (z,w)  := N (z,w) + \phi (z), \, \, z \in \C^n.
$$

We can prove a similar representation formula for the Monge-Amp\`ere measure of the twisted potential.
\begin{pro}
   If $\mu$ be a  probability measure on  $\mathbb{C}^{n}$ and $\phi \in DMA_{loc} (\C^n)$ then the complex Monge-Amp\`ere currents associated to $ V^{\phi}_{\mu}$ are given by the following formula: for any $1 \leq k \leq n$,
  $$
(dd^{c} V^{\phi}_{\mu})^{k}=\int_{(\mathbb{C}^{n})^k}dd^{c} N^{\phi} (.,w_{1})\wedge \cdots \wedge dd^{c}  N^{\phi} (.,w_{k})d\mu(w_{1}) \cdots d\mu(w_{k}),
   $$
   in the sense of $(k,k)$-currents on $\C^n$.
\end{pro}
The proof is the same as above.
\subsection{The regularizing property}
  We prove a regularizing property of the operator $V$ which generalizes and imporves a result of Carlehed \cite{Carlehed99}.
  Recall that a positive measure $\mu$ on $\C^n$ is said to have an atom at some point $a \in \C^n$ if $\mu (\{a\}) > 0$.  
   \begin{thm} \label{thm:abs}
   Let $\mu$ be a probability measure  on   $\mathbb{C}^{n}$ and $\phi \in DMA_{loc} (\C^n)$ which is smooth in some domain $B \subset \C^n$. Then   the Monge-Amp\`ere current $(dd^{c} V^{\phi}_{\mu})^{n}$ are absolutely continuous with respect to the Lebesgue measure on $B$ if and only if  $\mu$ has no atoms on $B$.
   \end{thm}
The proof of the "if part" of the theorem is based on the following lemma.

\begin{lem} Assume that  $\phi$ is smooth in some domain $B \subset \C^n$ and let $w_1, \cdots w_n \in \C^n$ be such that $w_{1}\neq w_{2}$. Then the Borel measure
 
$$
dd^{c} N^{\phi} (\cdot,w_1) \wedge \cdots \wedge dd^c  N^{\phi} (\cdot,w_n),
 $$
 is  absolutely continuous  with respect to the Lebesgue measure on $B$.
   \end{lem}
\begin{proof} The proof is based on an idea of  Carlehed \cite{Carlehed99}.
We are reduced to the proof of the following fact: let   $a, b_1, \dots, b_{m} \in \C^n$ with $1 \leq m \leq n - k$ such that $b_{j}\neq a $, for $1\leq j\leq m$, then the following current
$$
 \mu_{m,k} :=  \Big(dd^{c}  N^{\phi} (\cdot,a))\Big)^{k}\bigwedge_{1 \leq j \leq m} dd^{c}  N^{\phi} (\cdot,b_j)).
$$ 
is absolutely continuous with respect to the Lebesgue measure on $B$. 

Indeed, since the current $\mu_k$ is smooth  in $B \setminus \{a, b_1, \dots, b_{m} \}$, it is enough to show that $\mu_{m,k}$ puts no mass at the points $a, b_1, \dots, b_{m}$ that belong to $B$. Assume for simplicity that $a \in B$ and define the function
    $$
    u(z)=\log(|z-a|^{2}+|z\wedge a|^{2})+ \phi (z) + \displaystyle\sum_{j=1}^{m} N^{\phi} (z,b_j).
    $$
 Observe that $u$ is a plurisubharmonic function in $\C^n$  such that   $u \in DMA_{loc} (\C^n)$. 
 Since $u (z) \simeq \log(|z-a|^{2}+|z\wedge a|^{2})$ as $z \to a$, it follows from Demailly's comparison theorem (\cite{De93}) that
 
  $$
   \int_{\{a\}}(dd^{c} u)^{n}=\int_{\{a\}}\Big(dd^{c} \log(|\cdot -a|^{2}+|\cdot \wedge a|^{2})\Big)^{n}=1.
   $$ 

   On the other hand, performing the exterior product, we obtain
   
   \begin{eqnarray*}
   (dd^{c} u)^{n} \geq  \Big(dd^{c}\log(|\cdot -a|^{2}+|\cdot \wedge a|^{2})\Big)^{n}+ \mu_k  \geq  \delta_{a} + \mu_{m,k}, 
   \end{eqnarray*}
   in the weak sense on $\C^n$.  Therefore $ \mu_{m,k} (\{a\})  =0,$ which proves the required statement and the Lemma follows. \end{proof}

We now prove the theorem.

\begin{proof}
Assume first that the measure $\mu$ has an atom at some point $w \in \C^n$ then $\mu \geq c \delta_w$, where $c := \mu (\{w\}) > 0$.
Therefore from the definition we see that $\forall z \in \C^n$
$$
V^{\phi}_{\mu} (z) \leq c N (z,w) + c \phi (z).
$$

This implies that $V^{\phi}_\mu$ has a positive Lelong number at $w$ at least equal to $c$  and then the Monge-Amp\`ere mass  satisfies $(dd^c V^{\phi}_\mu)^n (\{w\}) \geq c^n > 0$ (see \cite{Ce04}).

 Assume now that $\mu$ has no atoms in $B$.
We want to show that $(dd^{c} V^{\phi}_{\mu})^{n}$ is absolutely continuous with respect to Lebesgue measure on $B$. Indeed, let $K \subset B$ be a compact set  such that $\lambda_{2 n} (K)=0$ and let us prove that $(dd^{c} {V}^{\phi}_{\mu})^{n}(K)=0$. 
 Set
 $$
 \Delta_B \, :=\{(w,...,w):w\in B\}\subset\mathbb{C}^{n}\times...\times\mathbb{C}^{n}.
 $$
 
  By Fubini theorem, denoting by $w = (w_1, \cdots,w_n) \in (\C^n)^n$, we get
 {\footnotesize
   \begin{eqnarray*}
&{}&\int_{K}(dd^{c} V^{\phi}_{\mu})^{n}\\
&=&\int_{K}\int_{(\mathbb{C}^{n})^n} dd^{c} N^{\phi} (\cdot,w_1)\wedge \cdots \wedge dd^{c} N^{\phi} (\cdot,w_n) d\mu^{(n)} (w)\\
&=&\int_{K}\int_{(\mathbb{C}^{n})^n \setminus\Delta_B} dd^{c} N^{\phi} (\cdot,w_1) \wedge \cdots \wedge dd^{c} N^{\phi} (\cdot,w_n) d\mu^{(n)}(w) \\
&&+\int_{K} \int_{\Delta_B} dd^{c} N^{\phi} (\cdot,w_1) \wedge...\wedge dd^{c}N^{\phi} (\cdot,w_n) d\mu^{(n)}(w).
   \end{eqnarray*}
   }
Since  $\mu$ puts no mass at any  point in $B$, it follows from Fubini's theorem that $\Delta_B$ has a zero measure with respect to the product measure  $ \mu^{(n)} = \mu\otimes...\otimes\mu$ on $(\C^n)^n$. Hence we have
{\footnotesize
 \begin{eqnarray*}
&{}&\int_{K}(dd^{c} V^{\phi}_{\mu})^{n}\\
&=&\int_{K}\Bigl(\int_{(\mathbb{C}^{n})^n \setminus\Delta_B}dd^{c} N^{\phi} (\cdot,w_1) \wedge...\wedge dd^{c}N^{\phi} (\cdot,w_n) d\mu^{(n)}(w_{1}, \cdots, w_{n})\Bigr)\\
&=&\int_{(\mathbb{C}^{n})^n\setminus\Delta_B}\int_{K}  dd^{c} N^{\phi} (\cdot,w_1) \wedge...\wedge dd^{c}N^{\phi} (\cdot,w_n) d\mu^{(n)}(w_{1}, \cdots, w_{n}).\\
   \end{eqnarray*}
   }
We  set  
$$
f(w_{1},...,w_{n})=\int_{K}dd^{c} N^{\phi} (\cdot,w_1) \wedge...\wedge dd^{c}N^{\phi} (\cdot,w_n).
$$
 Then  using the previous lemma,we see that  if $(w_1, \cdots,w_n) \notin \Delta_B$,  the measure   
 $$
 dd^{c} N^{\phi} (\cdot,w_1) \wedge...\wedge dd^{c}N^{\phi} (\cdot,w_n),
 $$
  is  absolutely continuous  with respect to the Lebesgue  measure on $B$. Hence $f (w_1,\dots, w_n) = 0$ if $(w_1, \cdots,w_n) \notin \Delta_B$ and then
  $$
  \int_{K}(dd^{c} V^{\phi}_{\mu} )^{n}= \int_{(\mathbb{C}^{n})^n\setminus\Delta_B}  f (w_1,\dots, w_n) d\mu^{(n)}(w_{1}, \cdots, w_{n}) = 0
  $$
  and the theorem is proved.
  \end{proof}

%\begin{rem} \label{rem:abs}  Observe that it is enough to assume in the previous theorem that $\phi \in DMA_{loc} %%(\C^n)$ and is locally bounded near the support of the measure $\mu$.%
%\end{rem}

\section{Proofs of Theorems}
 
 \subsection{Localization of the potential}
To study the projective logarithmic potential we will localize it in the affine charts and use the previous results. Let $(\chi_j)_{0 \leq j \leq n}$ a fixed partition of unity subordinated to the covering $(\mathcal U_j)_{0 \leq j \leq n}$.
 We define $m_j := \int \chi_j d  \mu$ and  $I = I_\mu := \{j \in \{0, \cdots, n\} ; m_j \neq 0\}$.
 Then  $I \neq \emptyset$ and for $j \in I$, the measure  $\mu_j := (1 \slash m_j) \chi_j \mu$  is a probability measure on $\mathbb P^n$ supported in the chart $\mathcal U_j$ and we have a the following convex decomposition of $\mu$
 $$
 \mu = \sum_{j \in I} m_j \mu_j.
 $$
 Therefore the potential $\G_\mu (\zeta) $ can be written as
 $$
  \G_\mu (\zeta) = \sum_{j \in I} m_j  \G_{\mu_j} (\zeta), \, \, \zeta \in \Pn,
 $$
 so that we are reduced to the case of a compact measure supported in an affine chart.

 Without loss of generality we may always assume the $\mu$ is compactly supported in $\mathcal U_0$.
 Then the potential $G_\mu$ can be written as follows
 $$
 \G_\mu (\zeta) := \int_{\mathcal U_0} G (\zeta,\eta) d \mu (\eta) = (1 \slash 2)  \int_{\mathcal U_0}   \log \, \frac{\vert \zeta \wedge \eta \vert^2}{\vert \zeta \vert^{2} \vert \eta \vert^2}  \, \, d \mu (\eta).
$$
 Since we, integrate on $ \mathcal U_0$ we have $\eta_0 \neq 0$ and we can use the affine coordinates
 $$
 w := (\eta_1 \slash \eta_0, \cdots, \eta_n \slash \eta_0).
 $$
 Therefore

 $$
 \vert \zeta \wedge \eta \vert^2 =  \sum_{1 \leq j \leq n}  \vert \zeta_0 w_j - \zeta_j \vert^2 +  \sum_{1 \leq i< j \leq n} \vert \zeta_i w_j - \zeta_j w_i\vert^2.
 $$
  Since the measure $\mu$ is supported on $\mathcal U_0$, the potential $\G_\mu (\zeta)$ is smooth outside the compact set $\text{Supp} \mu$ and we are reduced to the study of the potential $\G_\mu (\zeta)$ on  the open set  $\mathcal U_0$.

  The restriction of  $G (\zeta,\eta)$ to $\mathcal U_0 \times \mathcal U_0$ can be expressed in the affine coordinates
  as follows:

  Set
  $$
  z := (\zeta_1 \slash \zeta_0, \cdots, \zeta_n \slash \zeta_0);
  $$

  Then the kernel can be written as

\begin{eqnarray} \label{eq:projlogkernel}
G (\zeta,\eta) & = & (1 \slash 2) \log \, \frac{\vert z - w \vert^2 + \vert z \wedge w\vert^2}{(1 + \vert z \vert^{2}) (1 + \vert w \vert^2)} \\
& = & (1 \slash 2) \log \, \frac{\vert z - w \vert^2 + \vert z \wedge w\vert^2}{ 1 + \vert w \vert^2} - (1 \slash 2) \log (1 + \vert z \vert^{2})\\
& = & N (z,w) - (1 \slash 2) \log (1 + \vert z \vert^{2}), \nonumber
\nonumber
\end{eqnarray}
 where
   $$
   N (z,w) := (1 \slash 2) \log \, \frac{\vert z - w \vert^2 + \vert z \wedge w\vert^2}{ 1 + \vert w \vert^2}
   $$ 
  is the projective logarithmic kernel on $\C^n$ which was studied in the previous sections.

\subsection{Proof of Theorem A}
1. From the previous localization, it follows that for each $j \in I$, we have for $\zeta \in \mathcal{U}_j$
$$
 \G_{\mu_j} (\zeta) = V_{\mu_j} (z)  - (1 \slash 2) \log (1 + \vert z \vert^{2}),
$$
where $z := {\bf z}^j (\zeta)$ are the affine coordinates of $\zeta$ in $\mathcal{U}_j$.
By Theorem 2.10 and Theorem 5.2, it follows that for each $j \in I$, the function $ \G_{\mu_j}$ is $\omega_{FS}$-plurisubharmonic and  
$\Vert \nabla G_{\mu_j}\Vert \in L^p (\mathbb P^n)$ for any $0 < p < n$.
Therefore the convex combination $\G_{\mu}  = \sum_{j \in I} \G_{\mu_j} $ also satisfies the same properties i.e. $ \G_{\mu}$ is $\omega_{FS}$-plurisubharmonic and  
$\Vert \nabla G_{\mu}\Vert \in L^p (\mathbb P^n)$ for any $0 < p < n$.

To study the complex Hessian  $(\omega +dd^c G_\mu)^k $, it is enough to localize to a small ball $B \subset \mathcal U_j \simeq \C^n$ ($0 \leq j \leq n$) such that $\mu (B) > 0$.

 Then if we set  $\mu_B := {\bf 1}_B \mu$, we have
$$
\mu = s \mu_B + (1 - s) \nu,
$$
where $ 0 < s \leq 1$ is a positive number, $\mu_B $ is a probability measure supported on $\bar B$ and $\nu$ is a probability measure supported on the complement of $B$.

Therefore we have
$$
\G_\mu =  s \G_{\mu_B} + (1 - s) \G_{\nu},
$$
where $\G_{\nu}$ is a smooth $\omega$-psh function in $B$, since the support of $\nu$ is contained in the complement of $B$.

Then working on the coordinats $z$ in the chart $\mathcal U_j \simeq \C^n$, and setting $ \ell (z) := (1 \slash 2) \log (1 + \vert z\vert^2)$, the local potential of $\omega$ in $\mathcal U_j$,  we conclude that 
$$
G_\mu + \ell = s V_{\mu_B} + (1-s) V_\nu = V_{s \mu_B} + \phi,
$$
where $\phi$ is a plurisubaharmonic function in $\C^n$ which is smooth in $B$.

By Theorem \ref{thm:abs} (5.2) it follows that $ W_B := V_{s \mu_B} + \phi \in PSH (B) \cap  W^{2,p} (B)$ for any $0 < p < n$.
Let $1 \leq k \leq n - 1$. Since for any $1 \leq i, j \leq n$,
$$
\partial^2 W_B \slash \partial z_i \partial \bar{ z}_j \in L^p (B), \, \, \forall 0 < p < n,
$$ 
it follows by H\"older inequality that the coefficients of the current $(dd^c W_B)^k$ are  in $L^r (B)$  as far as  $r \geq 1$ and $1 \slash r = k \slash p$ i.e. $r = p \slash k \geq 1$. This is possible by choosing $p$ such that $ k < p < n$ since  $1 \leq k < n$.
Therefore $(dd^c W)^k \wedge \omega^{n-k}$ is absolutely continuous w.r.t. the Lebesgue measure on $B$ with a density in $L^r (B)$ for any $ 1 < r < n \slash k$.

On the other hand 
$$
G_\mu + \ell =  V_{s \mu_B} + \phi = V^{\phi}_{s \mu_B},
$$
is a twisted logarithmic potential.
Therefore by Theorem \ref{thm:abs}, it follows that 
$$
 (\omega + dd^c G_\mu)^n = (dd^c V_{s \mu_B}^{\phi})^n,
 $$
is absolutely continuous w.r.t. the Lebesgue measure on $B$.
This proves Theorem A.

\subsection{Proof of Theorem B}

We will use Lemma \ref{lem:Rieszpotential}. By the localization principle, it is enough to assume that $\mu$ is compactly supported in a chart $\mathcal U_j \simeq \C^n$. Then in the affine coordinates we have
$$
\G_\mu (\zeta) = V_\mu (z) - (1 \slash 2) \log (1 + \vert z\vert^2),
$$
in $\C^n$. 

From (\ref{eq:gradupperbound}) we see that locally $\vert \nabla V_\mu\vert $ is dominated by the Riesz potential $J_{\mu,1}$ and using the Lemma \ref{lem:Rieszpotential} with $\alpha = 1$ we obtain the first statement of the theorem about the gradient of $V_\mu$. The H\"older continuity of $\G_\mu$ follows then by the classical lemma of Sobolev-Morrey.

The second statement is proved in the same way.
Indeed we have $\omega + dd^c \G_\mu = dd^c V_\mu,$ and by (\ref {eq:secondorderest}), the coefficients of the current $dd^c V_\mu$ are locally dominated by the Riesz potential $J_{\mu,2}$.
Again using  Lemma \ref{lem:Rieszpotential} with $\alpha = 2$, we conclude that 
$$
dd^c V_\mu \in L^p_{2,loc} (\C^n), \, \, \forall  p < \frac{2 n -s}{(2 - s)_+}.
$$
the second statement of the theorem follows by  H\"older inequality.

\bigskip

\noindent{ \bf Aknowlegements :} 
{\it This paper is dedicated to the memory of Professor Ahmed Intissar who passed away in July 2017.  He was a talented mathematician who had a big scientific influence among the mathematical community in Morocco. We are grateful to him for his generosity, mathematical rigor and all what he taught us. 

\smallskip

We would like to thank Professor Vincent Guedj for many  interesting discussions and useful suggestions on the subject of this paper. 

We also thank the referee for his careful reading. By pointing out several typos and making useful suggestions, he helped to improve the presentation of the final version of this article.}

\newpage

 \end{document}